\documentclass[12pt]{amsart}

\usepackage[utf8]{inputenc}
\usepackage{amsmath}
\usepackage{amsfonts}
\usepackage{amssymb}
\usepackage{amsthm}
\usepackage{float}
\usepackage{mathtools}
\usepackage{csquotes}
\usepackage[numbers]{natbib}
\usepackage{pdfpages}

\RequirePackage{doi}

\usepackage[margin=3cm]{geometry}

\usepackage{tikz-cd}
\usetikzlibrary{arrows,matrix,shapes,decorations.text, mindmap,shapes.misc}
\usepackage{metalogo}
\newcommand{\C}{\mathbb{C}}

\newcommand{\T}{\mathbb{T}}
\renewcommand{\S}{\mathbb{S}}
\renewcommand{\P}{\mathbb{P}}
\newcommand{\R}{\mathbb{R}}

\newcommand{\ip}[2]{\left\langle #1,#2\right\rangle}

\newcommand{\ehn}{\mathcal{O}_{\C\P^{n-1}}(-n)}

\newcommand{\oa}{\overline{\alpha}}
\newcommand{\om}{\overline{\mu}}

\newtheorem{Thm}{Theorem}

\newtheorem{Lem}[Thm]{Lemma}
\newtheorem{Cor}[Thm]{Corollary}
\newtheorem{Prop}[Thm]{Proposition}
\usepackage{hyperref}
\hypersetup{colorlinks=true, linkcolor=violet, citecolor=blue}
\theoremstyle{definition}
\newtheorem{Def}[Thm]{Definition}

\newtheorem{Rem}[Thm]{Remark}

\usepackage[toc,page]{appendix}
\usepackage{dsfont}
\usepackage{authblk}

\usepackage[foot]{amsaddr}

\makeatletter
\@namedef{subjclassname@2020}{%
  \textup{2020} Mathematics Subject Classification}
\makeatother

\title{A detailed look at the Calabi-Eguchi-Hanson spaces}
\author{J\o rgen Olsen Lye}
\address{Institut f\"{u}r Differentialgeometrie, Leibniz Universit\"at Hannover}
\email{joergen.lye@math.uni-hannover.de}
\begin{document}
\pagestyle{plain}

\subjclass[2020]{Primary 32Q25; Secondary 53C42, 53C25\\ \indent \keywordsname : Calabi-Yau, K\"{a}hler-Einstein ALE, Minimal submanifolds.}

\phantomsection

\maketitle

\begin{abstract}
This article takes a detailed look at the Ricci-flat metrics introduced by Eguchi-Hanson and Calabi on the canonical line bundle of complex projective space. We give a description of these spaces as resolutions of certain orbifold singularities. We then compute the curvature explicitly and show that all compact, minimal submanifolds are contained in the zero section. This extends a result by Tsai and Wang. 
\end{abstract}

\section{Introduction}
Finding Einstein (pseudo-) Riemannian metrics has been an interesting project for more than a century. All the early successes rely on assuming a high degree of symmetry, reducing the Einstein equations to an ODE. This was successfully carried out by Eguchi and Hanson in \cite{EH}, allowing them to find a Ricci-flat, nowhere flat metric locally defined on $\R^4\setminus \{0\}$, which in fact is a complete hyperk\"{a}hler metric defined on the total space of the cotangent bundle $T^*\C\P^1=\mathcal{O}_{\C\P^{1}}(-2)$ of $\C\P^1$. The fact that the a priori real metric is K\"{a}hler and defined on a certain line bundle was quickly seen by Calabi \cite{Cal79} to not be a coincidence.   Calabi \cite{Cal79} then proceeded to generalize the Eguchi-Hanson metrics to two families of spaces, giving explicit examples of complete hyperk\"{a}hler metrics on $T^*\C\P^n$ and complete Ricci-flat K\"{a}hler metrics on $\mathcal{O}_{\C\P^{n-1}}(-n)$. We call the latter family Calabi-Eguchi-Hanson spaces in this article, even though the first family would also deserve this name. The Calabi construction was taken up again by R. Bryant and S. Salamon \cite{Bryant} and M. Stenzel \cite{Stenzel}, and they introduced Ricci-flat metrics on certain spin bundles and on $T^*\S^n$ respectively. In the Stenzel case $T^*\S^n$, the metrics are K\"{a}hler metrics.

A common feature of all these special holonomy spaces is that they are the total spaces of certain vector bundles over compact manifolds ($\S^n$ or $\C\P^n$, to be precise). These constructions were taken up again by B. Feix \cite{Feix} and D. Kaledin \cite{Kaledin} for arbitrary complex spaces $X$. They both independently showed that no matter which analytic metric on puts on $X$, one can extend it to hyperk\"{a}hler metric defined on a neighbourhood of the zero section of $T^*X$, but one cannot hope for completeness. In \cite{Bielawski} R. Bielawski proved that for any real analytic K\"{a}hler manifold $X$, one can extend the metric to be a Ricci-flat K\"{a}hler metric in a neighbourhood of the zero section in the canonical line bundle $\mathcal{K}_X$. This extends Calabi's second family of examples, $\mathcal{O}_{\C\P^{n-1}}(-n)=\mathcal{K}_{\C\P^{n-1}}$, but one cannot in general hope for completeness.

Apart from being explicit examples of non-flat, Ricci-flat spaces, the Calabi-Eguchi-Hanson spaces provide simple but interesting examples of instantons and asymptotically locally Euclidean (ALE) spaces. For $n=2$ they enter as crucial building blocks of Kummer K3 surfaces \cite{Page,GibbonsPope,Kob}, which has a generalization to $n=3$ \cite{Lu}. In both cases, one glues in copies of Calabi-Eguchi-Hanson spaces into a singular quotient of a torus. One equips the Calabi-Eguchi-Hanson spaces with the Ricci-flat metrics we will discuss in this article, which one glues together with a flat metric on the torus. If the parameter in the Calabi-Eguchi-Hanson metric is small enough, this procedure produces a compact, almost Ricci-flat Calabi-Yau manifold. The Ricci-flat metric of Yau \cite{Yau78} can be shown \cite{Kob,Lu} to be close to the given almost Ricci-flat metric. 


The aims of this article are twofold. The first goal is to present an interpretation of the spaces $\mathcal{O}_{\C\P^{n-1}}(-k)$ as smooth resolutions of $\C^n/\mu_k$, where $\mu_k$ denotes the $k$'th roots of unity acting diagonally. The second goal is to give a more detailed analytic description of the Ricci-flat Calabi-Eguchi-Hanson metric. We also include an addition to \cite[Theorem 5.3]{MeanCurv}. Much of the material is known in some form to the experts, but seems to be lacking a clear and explicit exposition. We hope the present article can serve as a gentle introduction to these fascinating spaces.

Calabi's deduction of the Ricci-flat metric in \cite{Cal79} is a bit different. He instead works directly on the bundle $p\colon \ehn\to \C\P^{n-1}$. On the trivializing sets $U_i$, he makes an ansatz for a K\"{a}hler potential on $\ehn_{\vert U_i}$ of the form
\[\Phi_i\coloneqq p^*\phi_{FS}+v(t),\]
where $\phi_{FS}$ is the Fubini-Study K\"{a}hler potential on $U_i\subset \C\P^{n-1}$, $t=\vert y\vert^2_{g_{FS}}$ is the norm squared of an element in the fibre, and $v$ is an unknown scalar function. Imposing Ricci-flatness, Calabi finds an ODE for $v$ which he explicitly solves. This leads to the metric expression \eqref{eq:EHPullback}.

Calabi's approach makes the completeness of the metric and the role of the Fubini-Study metric much more transparent. What we do in this paper is to start looking for rotationally symmetric K\"{a}hler potentials on $\C^n\setminus \{0\}$. These of course descend to quotients of $\C^n\setminus \{0\}$ by $\mu_k$ for any $k$, and extend to Ricci-flat metrics on $\mathcal{O}_{\C\P^{n-1}}(-k)$ if and only if $k=n$. These Ricci-flat metrics are exactly the same as the ones found by Calabi. We feel that the deduction presented in this paper offers a fruitful complementary access to these metrics.

\subsection{Statement of the results}
Definitions and proofs of the following results are given in the main text.

\noindent
\textbf{Notational:} 
We deal with K\"{a}hler manifolds in this article, and we will employ complex notation for tensors throughout. Good introductions can be found in \cite[Chapter IX.5]{KoNo} and \cite[p. 41 \& Chapter 4.2-4.4]{Kahler}.

Throughout the article, we will use $u\coloneqq \vert z\vert^2_{\C^n}$. Indices on $z\in \C^n$ are raised and lowered using the Euclidean metric. In particular, $\bar{z}_\mu\coloneqq \bar{z}^\nu \delta_{\mu\nu}$ etc. 
Let $\mu_k$ denote the group of $k$'th roots of unity. Consider the diagonal $\mu_k$ action on $\C^n$, 
\begin{equation}
\C^n \times \mu_k\to \C^n
\notag
\end{equation} 
given by 
\begin{equation}
(z,\zeta)\mapsto \zeta z.
\notag
\end{equation}
\begin{Thm}
\label{Thm:EH-Alg-Setup}
 For $k,n\geq 2$ the quotient $\C^n/\mu_k$ is singular, but the blow-up is a resolution 
\begin{equation}
\mathcal{O}_{\C\P^{n-1}}(-k)\cong Bl_0(\C^n/\mu_k)\cong Bl_0(\C^n)/\mu_k =\mathcal{O}_{\C\P^{n-1}}(-1)/\mu_k.
\end{equation}

For $n=k=2$, the space $\mathcal{O}_{\C\P^{1}}(-2)$ is the minimal resolution of $\C^2/\mu_2$.
\end{Thm}
\begin{Rem}
This result is well-known to the experts, but we do not know of any reference for all $k,n\geq 2$.
\end{Rem}
 
 \begin{Thm}[\cite{EH},\cite{Cal79}]
\label{Thm:EHMetric}
The spaces ${O}_{\C\P^{n-1}}(-k)$ admit Ricci-flat metrics if and only if $k=n$, in which case one can equip them with an asymptotically flat, Ricci-flat metric whose expression on $\mathcal{O}_{\C\P^{n-1}}(-n)\setminus \C\P^{n-1}\cong \left(\C^n\setminus \{0\}\right)/\mu_n$ is given by
\begin{equation}
g_{\mu\bar{\nu}}=\sqrt[n]{1+\frac{a^n}{u^n}}\left(\delta_{\mu\bar{\nu}}-\frac{a^n}{a^n+u^n} \frac{\bar{z}_\mu z_\nu}{u}\right),
\label{eq:GeneralEHMetric}
\end{equation}
for some parameter $a>0$.
The associated K\"{a}hler potential reads
\begin{equation}
f(u)=\sqrt[n]{a^n+u^n}+\frac{a}{n}\sum_{j=0}^{n-1} \zeta^j \log\left( \sqrt[n]{1+\frac{u^n}{a^n}}-\zeta^j\right)+const.,
\label{eq:GeneralEHPotential}
\end{equation}
 where $\zeta$ is a primitive $n$'th root of unity.
Restricting this Ricci-flat metric to the zero section $\C\P^{n-1}\subset \mathcal{O}_{\C\P^{n-1}}(-n) $ gives
\begin{equation}
ds^2_{\vert \C\P^{n-1}}=a\cdot  ds^2_{FS},
\label{eq:InducesFS}
\end{equation}
where $ds^2_{FS}$ is the Fubini-Study metric on $\C\P^{n-1}$, given in homogeneous coordinates $[z^1:\dots :z^n]$ by 

\begin{equation}
ds^2_{FS}\coloneqq \left(\frac{\vert dz \vert^2}{\vert z\vert^2}-\frac{\vert \bar{z}\cdot dz\vert^2}{\vert z\vert^4}\right).
\label{eq:GeneralFSMetric}
\end{equation}

\end{Thm}
 
\begin{Rem}
For $n=2$, the potential \eqref{eq:GeneralEHPotential} coincides with the one derived in \cite{EH}. According to \cite[Equation 5]{Joyce}, \eqref{eq:GeneralEHPotential} coincides with \cite[Equation 4.14]{Cal79}. 
\end{Rem}

\begin{Rem}
The spaces $\mathcal{O}_{\C\P^{n-1}}(-n)$ with metric given by \eqref{eq:GeneralEHMetric} are examples of ALE (asymptotically locally Euclidean) spaces, and Theorem \ref{Thm:EH-Alg-Setup} says that they are resolutions of $\C^n/\mu_n$. 

For $n=2$, \cite{Kronheimer1} and \cite{Kronheimer2} shows that \textit{any} ALE hyperk\"{a}hler manifold is the resolution on $\C^2/\Gamma$ for some finite subgroup $\Gamma\subset SU(2)$. The Eguchi-Hanson space $\mathcal{O}_{\C\P^1}(-2)$ is in this sense the simplest, non-trivial hyperk\"{a}hler ALE space. 
\end{Rem} 
 
\begin{Prop}
\label{Prop:EHCurvature}
Choose standard coordinates on $\C^n\setminus \{0\}$. Then the Christoffel symbols associated to the metric \eqref{eq:GeneralEHMetric} take the form
\begin{equation}
\Gamma_{\mu\alpha}^\lambda =-\frac{a^n}{u(a^n+u^n)}\left(\bar{z}_\mu \delta^\lambda{}_\alpha+\bar{z}_\alpha \delta^\lambda{}_\mu -(n+1)\frac{\bar{z}_\alpha\bar{z}_\mu}{u}z^\lambda\right)
\label{eq:GeneralChristoffel}
\end{equation}
The curvature tensor of \eqref{eq:GeneralEHMetric} reads
\begin{align}
R_{\mu\bar{\nu}\alpha\bar{\beta}}&=\frac{a^n}{(a^n+u^n)^\frac{n+1}{n}} \Bigg(g_{\alpha\bar{\nu}}g_{\mu\bar{\beta}}+g_{\mu\bar{\nu}}g_{\alpha\bar{\beta}}\notag \\
&-(n+1)\left(\frac{u}{\sqrt[n]{a^n+u^n}}\right)^{n-1}\left(\frac{\bar{z}_\mu z_\beta}{u}g_{\alpha\bar{\nu}}+\frac{\bar{z}_\alpha z_\beta}{u}g_{\mu\bar{\nu}}+\frac{\bar{z}_\mu z_\nu}{u}g_{\alpha\bar{\beta}}+\frac{\bar{z}_\alpha z_\nu}{u}g_{\mu\bar{\beta}}\right)\notag \\
&+(n+1)(n+2)\left(\frac{u}{\sqrt[n]{a^n+u^n}}\right)^{2n-2}\frac{\bar{z}_\mu z_\nu \bar{z}_\alpha z_\beta}{u^2}\Bigg).
\label{eq:EHRiemannTensor}
\end{align}

and the Kretschmann scalar (norm of the curvature tensor) is

\begin{equation}
K(z)\coloneqq R_{\mu\bar{\nu}\alpha\bar{\beta}}R^{\bar{\mu}\nu \bar{\alpha}\beta}(z)=\frac{a^{2n}}{(a^n+u^n)^{2\left(\frac{n+1}{n}\right)}} n(n+2)(n^2-1).
\label{eq:EHKretschmann}
\end{equation}
or, for $n=2$,
\begin{equation}
K(z)=\frac{24a^4}{(a^2+u^2)^3}.
\label{eq:EH-Scalar}
\end{equation}

\end{Prop}
\begin{Rem}
Again, these are both computed for $n=2$ in \cite{EH}, but we do not know of a reference for the higher dimensional results.
\end{Rem}

The nowhere vanishing, parallel holomorphic volume form is the same as for the flat metric:
\begin{Prop}
\label{Prop:HolomForm}
Let $\eta^{\C}\coloneqq dz^1 \wedge \dots \wedge dz^n$ be the standard holomorphic $n-$form on $\C^n$. This descends to $\left(\C^n\setminus \{0\}\right)/\mu_n$ and $\eta\coloneqq \pi^* \eta^{\C}$ extends to a parallel (with respect to the Ricci-flat metric), nowhere vanishing holomorphic volume form on $\ehn$.
\end{Prop}

\begin{Thm}
\label{Thm:Minimal}
Equip $\mathcal{O}_{\C\P^{n-1}}(-n)$ with the Ricci-flat metric described in Theorem \ref{Thm:EHMetric}. Then any compact, minimal submanifold is contained in $\C\P^{n-1}\subset \mathcal{O}_{\C\P^{n-1}}(-n)$.
\end{Thm}

\begin{Cor}
\label{Cor:NoClosedEH}
The only closed geodesics in $\mathcal{O}_{\C\P^{n-1}}(-n)$ are the geodesics in $(\C\P^{n-1},g_{FS})$, all of which are closed, and these are not stable. 
\end{Cor}
\begin{Rem}
Theorem \ref{Thm:Minimal} applies the trick of \cite[Theorem 5.3]{MeanCurv}, where Chung-Jun Tsai and Mu-Tao Wang prove a direct analogue of Theorem \ref{Thm:Minimal} for $T\S^n$ with the Stenzel metric, $T^*\C\P^n$ with the Calabi metric\footnote{ $\mathcal{O}_{\C\P^1}(-2)\cong T^*\C\P^1$, and the metrics are the same, but for $n>2$ the spaces considered in this article are different from $T^*\C\P^n$.}, and the total space of certain spinor bundles with the Bryant-Salamon metric. 

The $n=2$ case of the corollary is also covered by \cite{BY73}, which deals with arbitrary hyperk\"{a}hler 4-manifolds. \cite{BY73} applies since \eqref{eq:EH-Scalar} in particular says that $\mathcal{O}_{\C\P^1}(-2)$ with its Ricci-flat K\"{a}hler metric is nowhere flat.
\end{Rem}

 Much of this article is based on the author's PhD thesis \cite[Chapter 4]{PhD}, written under the supervision of Nadine Gro\ss e and Katrin Wendland. It grew out of an investigation of Kummer K3 surfaces, which is a companion paper \cite{KummerPaper}.

\section{Algebraic aspects} 
In this section we set out to prove Theorem \ref{Thm:EH-Alg-Setup}.

\begin{Lem}
\label{Lem:CnmodkAnalytic}
  The quotient spaces $\C^n/\mu_k$ are complex analytic varieties with isolated singularities at the origin, hence they are complex spaces. The subspaces $(\C^n\setminus\{0\})/\mu_k$ are complex manifolds. 
\end{Lem}
\begin{proof}
To see that $\C^n/\mu_k$ is a variety, one can argue as follows. 
For any finite group $G$ acting on an affine variety $X$, i.e. $X=Spm(R)$, we have that $X/G\cong Spm(R^G)$, where $Spm(R)$ has \textit{maximal} ideals as points (as opposed to $Spec(R)$, which has all prime ideals as points), and $R^G$ means the $G$-invariants of $R$. The ring $R^G$ is known to be finitely generated when $G$ is linearly reductive, which all finite groups are. A reference for these facts is for instance \cite[Theorem 5.9]{Mum03} combined with \cite[Corollary 5.17]{Mum03}. In our case, $X=\C^n=Spm(\C[z_1,\dots,z_n])$, so $\C^n/\mu_k\cong Spm(\C[z_1,\dots,z_n]^{\mu_k})$, which is a variety. 

To see that $\C^n/\mu_k$ is not smooth near the origin, notice that it is not even a topological manifold there. A way to see this is as follows. Assume for contradiction that there is some neighbourhood $U\subset \C^n/\mu_k$ containing $0$ such that $U\cong \C^n$. Any neighbourhood containing $0$ also contains the boundary of an orbiball $\partial B_\delta(0)/\mu_k\coloneqq \left\{[(z_1,\dots,z_n)]\in \C^n/\mu_k \, \big\vert \, \vert z\vert_{\C^n}=\delta\right\}$ for some $\delta>0$, which is a lens space of real dimension $2n-1$, hence cannot be embedded in codimension 1. This contradicts the existence of the neighbourhood $U$. 

When the origin is removed from $\C^n$, the $\mu_k$ action is free and proper. One can then cite a well-known result, for instance  \cite[Prop. 2.1.13]{Huybook}, to conclude that $(\C^n\setminus \{0\})/\mu_k$ is a complex manifold.
\end{proof}
\begin{Rem}
\label{Rem:C2mod2Analytic}
For the Eguchi-Hanson case, $k=n=2$, one can of course easily write out some more explicit details as follows. First of all, $\C[u,v]^{\mu_2}\xrightarrow{\cong} \C[x,y,z]/(xy-z^2)$, where the map sends $(u,v)\mapsto (u^2,v^2,uv)$. That the variety $\C^2/\mu_2\cong Spm(\C[x,y,z]/(xy-z^2))$ is singular at the origin can be directly checked using the Jacobi criterion.
\end{Rem}

\begin{Def}[Resolution of Singularities]
Let $X$ be a complex space with singular set $X_{sing}$. A resolution of $X$ is a complex manifold $\tilde{X}$ and a proper, bimeromorphic morphism $P\colon \tilde{X}\to X$ such that $P^{-1}(X_{sing})\subset \tilde{X}$ is a simple normal crossings divisor\footnote{Meaning the irreducible components of the divisor are smooth and intersect transversely, so this is in particular satisfied if $P^{-1}(X_{sing})$ is a complex submanifold.}, and 
\[P\colon \tilde{X}\setminus P^{-1}(X_{sing})\to X_0\]
is a biholomorphism.

A resolution $\tilde{X}\to X$ is \textit{minimal} if any other resolution of singularities $Q\colon Y\to X$, factors through $\tilde{X}$.  
\label{def:Res}
\end{Def}

\begin{Rem}
Resolutions of singularities always exits (see \cite{Hironaka} or, more specific to the analytic setting, \cite[Thm. 2.0.1]{Wlod} and \cite[Appendix]{Abhyankar}), but we will construct one explicitly for $\C^n/\mu_k$.
\end{Rem}
One easily checks that minimal resolutions, if they exist, are unique up to isomorphisms, so we will talk about \textit{the} minimal resolution (if it exists). We need to fix some standard notation.

\begin{Def}
\label{Def:TautLine}
Define the total space of the tautological line bundle over $\C\P^{n-1}$ as
\[\mathcal{O}_{\C\P^{n-1}}(-1)\coloneqq \left\{(z,\ell)\, \big\vert\, z\in \ell\right\}\subset \C^n\times \C\P^{n-1}.\]
The blow-up of the origin of $\C^n$ is defined as $Bl_0(\C^n)\coloneqq \mathcal{O}_{\C\P^{n-1}}(-1)$ with projection map $P\colon B_0(\C^n)\to \C^n$ defined by $P(z,\ell)=z$. 
\end{Def}

Next is a well-known result about extending group actions to the blow-up.
\begin{Lem}
\label{Lem:GActiononCn}
Let $G$ be any group acting on $\C^n$ via a representation $G\to GL(n,\C)$. Then one gets an induced action on $Bl_0(\C^n)$ by $G$ given by $g\cdot (z,\ell)\coloneqq (g\cdot z,g\cdot \ell)$, making the following diagram commute.

\begin{figure}[H]
\centering
\begin{tikzpicture}
  \matrix (m) [matrix of math nodes,row sep=3em,column sep=4em,minimum width=2em]
  {
     Bl_0(\C^n) & Bl_0(\C^n)/G \\
     \C^n & \C^n/G& {} \\};
  \path[-stealth]
    (m-1-1) edge node [below] {$\tilde{Q}$} (m-1-2)
            edge node [left] {$P$} (m-2-1)
    (m-1-2) edge node [right] {$\tilde{P}$}(m-2-2)
    (m-2-1) edge node [above] {$Q$} (m-2-2);
\end{tikzpicture}
\end{figure}


\end{Lem}
\begin{proof}
Since $G$ acts via linear transformations, $G$ will take lines to lines. So if $\ell$ is some line through the origin, then $g\cdot \ell$ is as well. This means we get a well-defined group action
\[G\times Bl_0(\C^n)\to Bl_0(\C^n),\]
\[(g,(z,\ell))\mapsto (g\cdot z,g\cdot \ell).\]
The commutative diagram is then evident. 
\end{proof}
We are finally ready to attack Theorem \ref{Thm:EH-Alg-Setup}.
\begin{proof}[Proof of Theorem \ref{Thm:EH-Alg-Setup}]

The action of $\mu_k$ on $\C^n$ is via linear transformations, so by Lemma \ref{Lem:GActiononCn} it induces an action on $\mathcal{O}_{\C\P^{n-1}}(-1)$ as follows.
\[\mu_k\times \mathcal{O}_{\C\P^{n-1}}(-1)\to \mathcal{O}_{\C\P^{n-1}}(-1)\]
\[(\zeta,(z,\ell))\mapsto (\zeta\cdot z,\zeta \cdot \ell)= (\zeta\cdot z,\ell).\]

From Lemma \ref{Lem:GActiononCn}, we also get the following commutative diagram.

\begin{figure}[H]
\centering
\begin{tikzpicture}
  \matrix (m) [matrix of math nodes,row sep=3em,column sep=4em,minimum width=2em]
  {
     \mathcal{O}_{\C\P^{n-1}}(-1) & \mathcal{O}_{\C\P^{n-1}}(-1)/\mu_k \\
     \C^n & \C^n/\mu_k& {} \\};
  \path[-stealth]
    (m-1-1) edge node [below] {$\tilde{Q}$} (m-1-2)
            edge node [left] {$\pi$} (m-2-1)
    (m-1-2) edge node [right] {$\tilde{\pi}$}(m-2-2)
    (m-2-1) edge node [above] {$Q$} (m-2-2);
\end{tikzpicture}
\end{figure}

The blow-down map is biholomorphisms away from the zero section $\C\P^{n-1}$,  $\pi\colon \mathcal{O}_{\C\P^{n-1}}(-1)\setminus \C\P^{n-1}\xrightarrow{\cong} \C^n\setminus \{0\}$, which gives an isomorphism
\[\tilde{\pi}\colon \left(\mathcal{O}_{\C\P^{n-1}}(-1)\setminus \C\P^{n-1}\right)/\mu_k\xrightarrow{\cong} \left(\C^n\setminus \{0\}\right)/\mu_k.\]
Furthermore, $\mu_k$ acts freely and transitively on $\C^n\setminus \{0\}$, so $\left(\C^n\setminus \{0\}\right)/\mu_k$ is a smooth manifold. It thus remains to show that $\mathcal{O}_{\C\P^{n-1}}(-1)/\mu_k$ is smooth and that $\tilde{\pi}^{-1}(\{0\})$ is a complex submanifold. We will do this by exhibiting an isomorphism
\[\mathcal{O}_{\C\P^{n-1}}(-1)/\mu_k \xrightarrow{\cong} \mathcal{O}_{\C\P^{n-1}}(-k)\]
and showing that $\tilde{\pi }^{-1}(\{0\})\cong \C\P^{n-1}$.

We use the description $\mathcal{O}_{\C\P^{n-1}}(-k)\coloneqq \left\{(z,\mathbf{\xi})\, \Big \vert\, z^i (\xi^j)^k = z^j (\xi^i)^k\right\} \subset \C^n\times \C\P^{n-1}$. We can define a map
\[f\colon \mathcal{O}_{\C\P^{n-1}}(-1)/\mu_k\to \mathcal{O}_{\C\P^{n-1}}(-k)\]
by specifying what happens when restricting both bundles\footnote{The total space $\mathcal{O}_{\C\P^{n-1}}(-k)$ can be thought of as a line bundle over $\C\P^{n-1}$, and that is the view taken in this proof.} to $U_i\coloneqq \{\xi_i\neq 0\}\subset \C\P^{n-1}$. Let $\zeta^j\coloneqq \frac{\xi^j}{\xi^i}$ on $U_i$. Write $([z^1],\dots ,[z^n])$ for a point in $\C^n/\mu_k$. We have
\begin{equation}
\left(\mathcal{O}_{\C\P^{n-1}}(-1)/\mu_k\right)_{\vert U_i}\cong \left\{\left(\zeta^1 [z^i],\zeta^2 [z^i],\dots ,\zeta^{i-1} [z^i], [z^i], \zeta^{i+1}[z^i],\dots \zeta^n [z^i]\right)\, \Big \vert \, z^i, \zeta^j\in \C\right\}
\notag
\end{equation}
and
\begin{equation}
\left(\mathcal{O}_{\C\P^{n-1}}(-k)\right)_{\vert U_i}\cong \left\{\left((\zeta^1)^k z^i,(\zeta^2)^k z^i,\dots ,(\zeta^{i-1})^k z^i, z^i, (\zeta^{i+1})^k z^i,\dots (\zeta^n)^k z^i\right)\,\Big\vert \,  z^i, \zeta^j\in \C\right\}.
\notag
\end{equation}
Define $f_i\colon \left(\mathcal{O}_{\C\P^{n-1}}(-1)/\mu_k\right)_{\vert U_i}\to \left(\mathcal{O}_{\C\P^{n-1}}(-k)\right)_{\vert U_i}$ by 
\begin{equation}
f_i\left(w^1,\dots ,w^n\right)=\left((w^1)^k ,\dots, (w^n)^k\right).
\notag
\end{equation}
The important thing happens at slot $i$, where $[z^i]\mapsto (z^i)^k$, which is clearly surjective, and it is injective since if $z^k=w^k$, then $\exists\, \zeta\in \mu_k$ such that $z=\zeta w\implies [z]=[w]$.

Under $f$, we see that $\{([0],\ell)\in \C^n/\mu_k\times \C\P^{n-1}\}\subset \mathcal{O}_{\C\P^{n-1}}(-1)/\mu_k$ corresponds to $\{(0,\ell)\in \C^n\times \C\P^{n-1}\}\cong \C\P^{n-1}\subset \mathcal{O}_{\C\P^{n-1}}(-k)$. So the fibre over $\{0\}$ is a complex submanifold.

Finally, to prove the fact that $\mathcal{O}_{\C\P^1}(-2)\to \C^2/\mu_2$ is the minimal resolution one can use a neat little criterion, namely Lemma \ref{Lem:TrivialCanonicalMin}, along with the fact that the spaces $\mathcal{O}_{\C\P^{n-1}}(-n)$ have trivial canonical bundle, which follows from the adjunction formula as we show below as Corollary \ref{Cor:Trivial}.

\end{proof}
\begin{Rem}
A couple of things should be pointed out at this point. Firstly, it is important to note that although the blow-down map $P\colon \mathcal{O}_{\C\P^{n-1}}(-1)\to \C^n$ is just a linear projection, the map $\mathcal{O}_{\C\P^{n-1}}(-k)\to \C^n/\mu_k$ is not linear, but behaves instead like a $k$'th root.

Secondly, the $\mu_k$-action on $\mathcal{O}_{\C\P^{n-1}}(-1)$ is not free, so the meaning of $\mathcal{O}_{\C\P^{n-1}}(-1)/\mu_k\cong \mathcal{O}_{\C\P^{n-1}}(-k)$ is that the left hand space is given a holomorphic structure by the identification with the right hand space, which is a complex manifold. This is analogous to identifying $\T^2/\mu_2$ with $\C\P^1$ via the Weierstra\ss{} $\wp$-function, which makes $\T^2/\mu_2$ into a Riemann surface even though the $\mu_2$-action has 4 fixed points.
\end{Rem}

\begin{Lem}
\label{Lem:TrivialCanonicalMin}
Assume $X\to Y$ is a resolution where $X$ is a non-singular, complex surface and $Y$ is potentially singular. If the canonical line bundle is trivial, $\mathcal{K}_X\cong \underline{\C}$, then $X\to Y$ is the minimal resolution.
\end{Lem}

\begin{proof}
By \cite[(4.3) Lemma, p. 98]{Surfaces} and their definition of a minimal resolution, \cite[p. 106]{Surfaces}, a non-singular, complex surface $X$ is the minimal resolution of $X\to Y$ for some singular, complex surface $Y$ if and only if $X$ does not contain any $(-1)$-curves. By \cite[(2.2) Prop., p. 91]{Surfaces}, a curve $C\subset X$ is a $(-1)$-curve if and only if $C^2<0$ and $(\mathcal{K}_X,C)<0$, both equations meaning intersection pairing. Since $\mathcal{K}_X\cong \underline{\C}$ by assumption, $(\mathcal{K}_X,C)<0$ cannot hold.\footnote{This argument is standard. Indeed, \cite{Kronheimer1} merely remarks that his resolutions of $\C^2/\Gamma$, for finite groups $\Gamma\subset SU(2)$, have $c_1(X)=0$ since they are hyperk\"{a}hler, and thus have to be the minimal resolution.}
\end{proof}

\label{Section:EHMetricProof}
We again need some preliminary results and definitions. We will divide the proof into two parts for readability. Let us first go through a standard argument to see that $\mathcal{O}_{\C\P^{n-1}}(-k)$ cannot admit a Ricci-flat metric if $k\neq n$. The idea is: if $X\coloneqq \mathcal{O}_{\C\P^{n-1}}(-k)$ admits a Ricci-flat metric, then Chern-Weil theory says that $c_1(X)=0$, which means $\mathcal{K}_X=\underline{\C}$, which will contradict the adjunction formula for $\mathcal{K}_X$ using $Y=\C\P^{n-1}\subset X$ as submanifold unless $k=n$. The details are as follows. First recall what the adjunction formula says.

\begin{Prop}
Let $Y\xrightarrow{\iota} X$ be a submanifold of a complex manifold $X$. Let $\mathcal{N}_{Y/X}$ denote the normal bundle of $Y$ in $X$. Let $\mathcal{K}_X$ and $\mathcal{K}_Y$ be the canonical (aka. determinant) line bundles of $X$ and $Y$ respectively. Then the adjunction formula is the statement that 
\begin{equation}
\mathcal{K}_Y\cong \iota^*\left(\mathcal{K}_X\right)\otimes \det\left(\mathcal{N}_{Y/X}\right).
\label{eq:Adjunction}
\end{equation}
\end{Prop} 
See \cite[Prop. 2.2.17]{Huybook} for a proof.
\begin{Lem}
Let $E\to X$ be a vector bundle and $\iota\colon X\to E$ the inclusion as the zero section. Then
\begin{equation}
\mathcal{N}_{X/E}\cong E.
\notag
\end{equation}
\end{Lem}

\begin{Cor}
\label{Cor:Trivial}
For $Y=\C\P^{n-1}$ and $X=\mathcal{O}_{\C\P^{n-1}}(-k)$, the adjunction formula, \eqref{eq:Adjunction}, says that
\begin{equation}
\mathcal{O}_{\C\P^{n-1}}(-n)\cong \iota^*\left(\mathcal{K}_X\right)\otimes \mathcal{O}_{\C\P^{n-1}}(-k).
\label{eq:SpecialAdj}
\end{equation}
\end{Cor}
\begin{proof}
The previous lemma says $\mathcal{N}_{Y/X}\cong \mathcal{O}_{\C\P^{n-1}}(-k)$. This is a line-bundle, hence equals its determinant line-bundle; $\det\left(\mathcal{N}_{Y/X}\right)=\mathcal{N}_{Y/X}$. It is a well-known fact that $\mathcal{K}_{\C\P^{n-1}}\cong \mathcal{O}_{\C\P^{n-1}}(-n)$ (see for instance \cite[Prop. 2.4.3]{Huybook}). Putting this into the adjunction formula, \eqref{eq:Adjunction}, gives the conclusion.
\end{proof}

\begin{Prop}
$X\coloneqq \mathcal{O}_{\C\P^{n-1}}(-k)$ does not admit a Ricci-flat metric for $k\neq n$.
\end{Prop}
\begin{proof}
Assume for contradiction that $X$ carries a Ricci-flat metric. Since $X$ deformation retract onto $\C\P^{n-1}$, we know that $H^1(X;\C)=0$, so also $H^1(X,\mathcal{O}_X)=0$. Then $\mathcal{K}_X\cong \underline{\C}$ as holomorphic line bundles, contradicting \eqref{eq:SpecialAdj} unless $k=n$, since $\mathcal{O}_{\C\P^{n-1}}(-n)\cong \mathcal{O}_{\C\P^{n-1}}(-k)$ if and only if $k=n$.
\end{proof}

We end this section by giving an explicit section of $\mathcal{K}_X\cong \underline{\C}$ when $X=\mathcal{O}_{\C\P^{n-1}}(-n)$.
\begin{Prop}
\label{Prop:HolomFormExtends}
Let $\pi\colon \mathcal{O}_{\C\P^{n-1}}(-n)\to \C^n/\mu_n$ be the resolution described in Theorem \ref{Thm:EH-Alg-Setup}, $n\geq 2$. Let $\eta^{\C}\coloneqq dz^1\wedge dz^2\wedge\dots \wedge dz^n$ denote a nowhere vanishing, holomorphic $n$-form on $\C^n$. Then $\eta^{\C}$ descends to a holomorphic $n$-form (also denoted $\eta^{\C}$) on $\left(\C^n\setminus \{0\}\right)/\mu_n$, and this $n$-form extends to a nowhere vanishing, holomorphic $n$-form $\eta$ on all of $ \mathcal{O}_{\C\P^{n-1}}(-n)$.
\end{Prop}

\begin{proof}
That $\eta^{\C}$ descends to $\left(\C^n\setminus \{0\}\right)/\mu_n$ follows from $\mu_n$ acting as a $SU(n)$-representation.

Let $U_i\coloneqq \left\{(\xi^1\colon \dots \colon \xi^n)\, \Big \vert \, \xi^i\neq 0\right\}\subset \C\P^{n-1}$ for $1\leq i\leq n$. Then, as in the proof of Theorem \ref{Thm:EH-Alg-Setup},
\begin{equation}
\mathcal{O}_{\C\P^{n-1}}(-n)_{\vert U_i}\cong \left\{\left( (\zeta^1)^nz^i,\dots,(\zeta^{i-1})^n z^i,z^i, (\zeta^{i+1})^n z^i,\dots (\zeta^n)^n z^i)\right)\, \Big\vert \, z^i,\zeta^j \in \C\right\},
\notag
\end{equation} 
with $\zeta^k\coloneqq \frac{\xi^k}{\xi^i}$.
The blow-down map $\pi\colon \mathcal{O}_{\C\P^{n-1}}(-n)_{\vert U_i}\to \C^n/\mu_n$ has the description
\begin{align}
&\pi\left(\left((\zeta^1)^n z^i,\dots,(\zeta^{i-1})^n z^i,z^i, (\zeta^{i+1})^n z^i,\dots (\zeta^n)^n z^i)\right)\right)\notag \\ &=\left[\left(\sqrt[n]{z^i}\zeta^1,\dots \sqrt[n]{z^i} \zeta^{i-1}, \sqrt[n]{z^i}, \sqrt[n]{z^i} \zeta^{i+1}, \dots, \sqrt[n]{z^i} \zeta^{n}\right)\right].
\notag
\end{align}
The pullback of $\eta^{\C}$ by $\pi$ can be described in two steps.
\begin{equation}
\pi^*(dz^k)=\begin{cases} \frac{1}{n} \zeta^k (z^i)^{\frac{1}{n}-1} dz^i+\sqrt[n]{z^i} d \zeta^k & i\neq k\\ \frac{1}{n} (z^i)^{\frac{1}{n}-1} dz^i & i=k,\end{cases}
\notag
\end{equation}
so
\begin{equation}
\pi^*(\eta^{\C})=\bigwedge_{k=1}^n \pi^*(dz^k)=\frac{1}{n} d\zeta^1\wedge d\zeta^2 \wedge \dots d\zeta^{i-1}\wedge dz^i \wedge d\zeta^{i+1}\wedge\dots \wedge d\zeta^n.
\label{eq:PullbackofEta}
\end{equation}
From \eqref{eq:PullbackofEta}, one sees that $\eta=\pi^*(\eta^\C)$ is a nowhere vanishing holomorphic $n$-form on $\mathcal{O}_{\C\P^{n-1}}(-n)_{\vert U_i}$. Since $i$ was arbitrary, this shows that $\eta^\C$ pulls back to a nowhere vanishing holomorphic $n$-form on all of $\mathcal{O}_{\C\P^{n-1}}(-n)$.
\end{proof}

\section{Metric aspects}
Here we prove Theorem \ref{Thm:EHMetric}, Proposition \ref{Prop:EHCurvature} and Theorem \ref{Thm:Minimal}.

It remains to deduce that $\mathcal{O}_{\C\P^{n-1}}(-n)$ indeed possesses the Ricci-flat metric claimed in Theorem \ref{Thm:EHMetric}. We will show how this comes about by starting with the ansatz that one has a rotationally symmetric K\"{a}hler potential.

\begin{Lem}
\label{Lem:RotSymMetric}
Assume a K\"{a}hler metric on $\C^n\setminus\{0\}$ is given by
\[g_{\mu\bar{\nu}}=\partial_\mu \partial_{\bar{\nu}} f,\]
and assume that the K\"{a}hler potential $f$ is $SO(2n)$-symmetric. Then, with $u\coloneqq r^2\coloneqq \sum_{\mu=1}^n z^\mu \bar{z}^\mu $, we have that
\begin{equation}
g_{\mu\bar{\nu}}=\delta_{\mu\nu} f'(u)+\overline{z}_\mu z_\nu f''(u).
\label{eq:SymMetric}
\end{equation}
This may also be written
\begin{equation}
g=\mathds{1}f'(u)+\overline{z}\otimes z f''(u).
\notag
\end{equation}
Furthermore,
\begin{equation}
\det(g)=f'(u)^{n-1}(uf'(u))'.
\label{eq:SymDet}
\end{equation}
\end{Lem}

\begin{proof}
The first bit is a trivial computation. For the determinant, notice that we may write the determinant as\footnote{Note that $f'(u)$ is an eigenvalue of $g$ which we assume is a metric, hence $f'(u)$ cannot be 0.}
\[\det(g)=f'(u)^n \det\left(\delta_{\mu\bar{\nu}}+\frac{f''(u)}{f'(u)}\overline{z}_\mu z_\nu\right).\]
Define the hermitian $n\times n$ matrix
\[B_{\mu\bar{\nu}}\coloneqq \frac{f''(u)}{f'(u)}\overline{z}_\mu z_\nu,\]
or, equivalently,
\[B=\frac{f''(u)}{f'(u)}\overline{z}\otimes z.\]
The matrix $B$ has eigenvalue $u\frac{f''(u)}{f'(u)}$ with multiplicity 1 and eigenvalue 0 with multiplicity $n-1$, as is seen by choosing eigenvectors parallel and orthogonal (with respect to Euclidean norm on $\C^n$) to any given $z$. From this little discussion, it follows that
\[\det(\mathds{1}+B)=\left(1+u\frac{f''(u)}{f'(u)}\right),\]
and thus
\[\det(g)=f'(u)^n \det(\mathds{1}+B)=f'(u)^n\left(1+u\frac{f''(u)}{f'(u)}\right),\]
which one can of course write
\begin{align}
\det(g)=&f'(u)^n\left(1+u\frac{f''(u)}{f'(u)}\right)\notag \\=&f'(u)^{n-1}\left(f'(u)+uf''(u)\right)\notag \\=&f'(u)^{n-1} \left(uf'(u)\right)'.
\notag
\end{align}
\end{proof}

The requirement that the Ricci-curvature vanishes will determine the function $f$ up to a non-trivial constant. We will explain this with two lemmas.
\begin{Lem}
\label{lem:ConstDet}
Assume we have a K\"{a}hler metric $g$ on $\C^n\setminus \{0\}$ with K\"{a}hler potential $f$, and assume that $f$ is spherically symmetric. Then $g$ is Ricci-flat if and only if the $det(g)$ is constant, which one might assume is 1 after scaling.
\end{Lem}
\begin{proof}
It's well known that on a K\"{a}hler manifold, one may write the Ricci curvature locally as 
\begin{equation}
R_{\mu\bar{\nu}}=-\partial_\mu \partial_{\bar{\nu}} \ln \left(\det(g)\right).
\notag
\end{equation}
See for instance \cite[Equation 4.63]{Kahler} for this fact. 

Since the potential is assumed to be spherically symmetric, \eqref{eq:SymDet} tells us that the determinant of the metric is spherically symmetric. Hence the Ricci-tensor takes the form
\begin{equation}
R_{\mu\bar{\nu}}=-\left(\delta_{\mu \bar{\nu}}\frac{d}{du} \ln\left(\det(g)\right)+\overline{z}_\mu z_\nu \frac{d^2}{du^2}\left(\ln\left(\det(g)\right)\right)\right).
\label{eq:RotSymRicci}
\end{equation}
One direction is clear; if $\det(g)$ is constant, the Ricci curvature vanishes.

For the converse, observe that \eqref{eq:RotSymRicci} has eigenvalue $\frac{d}{du} \ln\left(\det(g)\right)$ (with multiplicity $(n-1)$), hence $R_{\mu\bar{\nu}}=0$ implies $\frac{d}{du} \ln\left(\det(g)\right)=0$.

\end{proof}

\begin{Lem}[{\cite{Cal79}}]
\label{Lem:Ricci-FlatMetric}
Assume, as before, that we have a K\"{a}hler metric $g$ on $\C^n\setminus \{0\}$ with K\"{a}hler potential $f$, and assume that $f$ is spherically symmetric. Assume in addition that the metric is Ricci-flat and non-flat. Then the metric takes the form
\begin{equation}
g_{\mu\bar{\nu}}=\left(1+\left(\frac{a}{u}\right)^n\right)^{\frac{1}{n}}\left(\delta_{\mu\bar{\nu}}-\frac{a^n}{a^n+u^n}\frac{\bar{z}_\mu z_\nu}{u}\right)
\label{eq:Ricci-FlatMetric}
\end{equation}
for some constant $a>0$.
In terms of the potential $f$, we have
\begin{equation}
f'(u)=\left(1+\left(\frac{a}{u}\right)^n\right)^{\frac{1}{n}},
\label{eq:PotDer}
\end{equation}
which integrates to
\begin{equation}
f(u)=\sqrt[n]{a^n+u^n}+\frac{a}{n}\sum_{j=0}^{n-1} \zeta^j \log\left( \sqrt[n]{1+\frac{u^n}{a^n}}-\zeta^j\right)+const.
\label{eq:Calabi-Pot}
\end{equation}
for some primitive $n$'th root of unity $\zeta$.
\end{Lem}

\begin{proof}
From Lemma \ref{lem:ConstDet}, our metric is Ricci-flat if and only if, $\det(g)=1$. Inserting this into \eqref{eq:SymDet} gives us the following ODE to solve
\begin{equation}
f'(u)^{n-1}(uf'(u))'=1.
\label{eq:DetODE}
\end{equation}
This is readily solved by multiplying each side by $u^{n-1}$ and integrating in $u$, yielding
\[(uf'(u))^n=u^n+a^n\]
as a first step, where $a^n$ is some constant of integration. This solves\footnote{We remark again that $f'$ is an eigenvalue of $g$, which we want to be a metric. As such, one has to take the real, positive root when solving for $f'(u)$.} for $f'(u)$ as 
\[f'(u)=\left(1+\left(\frac{a}{u}\right)^n\right)^{\frac{1}{n}}.\]

To see that the function $f(u)$ given in \eqref{eq:GeneralEHPotential} is the correct integral of $f'(u)$ requires a little computation.
Introduce $x=\frac{u}{a}$, and observe that $f'(u)=\frac{1}{a}\frac{d}{dx}f(x)$, and that
\begin{equation}
\frac{f(x)}{a}=\sqrt[n]{1+x^n}+\frac{1}{n}\sum_{j=0}^{n-1}\zeta^j \log\left( \sqrt[n]{1+x^n}-\zeta^j\right).
\notag
\end{equation}
The derivative of this is easily computed to be
\begin{equation}
\frac{1}{a}f'(x)=\frac{x^{n-1}\sqrt[n]{x^n+1}}{x^n+1}\left(1+\frac{1}{n}\sum_{j=0}^{n-1} \frac{\zeta^j}{\sqrt[n]{x^n+1}-\zeta^j}\right).
\label{eq:DerivStep}
\end{equation}
To progress further, we need the fact that for any complex number $\alpha$ with $\vert\alpha\vert \neq 1$, and integer $n\geq 
1$, the following identity holds.
\begin{equation}
\frac{1}{n}\sum_{j=0}^{n-1} \frac{\zeta^j}{\alpha-\zeta^j}=\frac{1}{\alpha^n-1}.
\end{equation}
Inserting this with $\alpha\coloneqq \sqrt[n]{1+x^n}$ into \eqref{eq:DerivStep} yields
\[f'(u)=\sqrt[n]{1+x^{-n}}=\sqrt[n]{1+\frac{a^n}{u^n}},\]
which proves the claim.
\end{proof}

For many purposes, the metric is best expressed as a metric on $\C^n\setminus \{0\}$ or its quotient by $\mu_n$. To study the behaviour near the blown-up point, the standard coordinates on $\C^n$ are no longer suitable, and one should instead pull back the metric to $\mathcal{O}_{\C\P^{n-1}}(-n)$ using the blow-down map $\pi$. Using the coordinates from the standard trivialization $\mathcal{O}_{\C\P^{n-1}}(-n)_{\vert U_i}\cong U_i\times \C\cong \C^n$, we get tractable expressions for the metric which both shows that the metric extends across the zero section $\C\P^{n-1}$ and that $g_{\vert \C\P^{n-1}}=a \cdot g_{FS}$. The details are as follows.

\begin{Lem}
\label{Prop:EHInducesFS}
Identify $\left(\C^n\setminus\{0\}\right)/\mu_n$ as $\mathcal{O}_{\C\P^{n-1}}(-n)\setminus \C\P^{n-1}$ (see Theorem \ref{Thm:EH-Alg-Setup}), and give it the Ricci-flat K\"{a}hler metric described as \eqref{eq:GeneralEHMetric},
\begin{equation}
g_{\mu\bar{\nu}}=\left(1+\left(\frac{a}{u}\right)^n\right)^{\frac{1}{n}}\left(\delta_{\mu\bar{\nu}}-\frac{a^n}{a^n+u^n}\frac{\bar{z}_\mu z_\nu}{u}\right).
\notag
\end{equation}
Then this metric extends to all of $\mathcal{O}_{\C\P^{n-1}}(-n)$, where $u=0$ corresponds to the zero-section $\C\P^{n-1}\subset \mathcal{O}_{\C\P^{n-1}}(-n)$, and the Ricci-flat metric induces a metric on $\C\P^{n-1}$ which is precisely the parameter $a$ times the Fubini-Study metric.
\end{Lem}

\begin{proof}
Indeed, let $U_i=\{\xi_i\neq 0\}$ as before and let $\zeta_k\coloneqq \frac{\xi_k}{\xi_i}$. We write a point in $\mathcal{O}_{\C\P^{n-1}}(-n)_{\vert U_i} \cong \C^n$ as $(z\zeta_1^n,\dots ,z \zeta_{i-1}^n, z,z \zeta_{i+1}^n ,\dots z \zeta_n^n)$. The blow-down map then takes the form
\[\pi_i((z\zeta_1^n,\dots ,z \zeta_{i-1}^n, z,z \zeta_{i+1}^n ,\dots z \zeta_n^n))=(\sqrt[n]{z}\zeta_1,\dots ,\sqrt[n]{z},\dots,\sqrt[n]{z}\zeta_n).\]
We claim that the Calabi-Eguchi-Hanson metric takes the form
\begin{align}
\pi_i^*g_{EH}&= \left(\frac{1+\vert \zeta\vert^2}{\sqrt[n]{a^n+u^n}}\right)^{n-1}\left(\frac{(1+\vert \zeta\vert^2)}{n^2} \vert dz\vert^2+\vert z\vert^2 \vert d\zeta\vert^2+\frac{1}{n}\sum_k (\overline{z} \zeta_k dz d\overline{\zeta}_k+z\overline{\zeta}_k d\overline{z} d\zeta_k\right)\notag \\
&+\left(\frac{a}{\sqrt[n]{a^n+u^n}}\right)^{n-1} a g_{FS}.
\label{eq:EHPullback}
\end{align}
Here
\[\vert \zeta\vert^2=\sum_{k\neq i} \vert \zeta_k\vert^2,\]
\[\vert d\zeta\vert^2=\sum_k d\zeta_k d\overline{\zeta}_k,\]
and
\[u=\vert z\vert^{2/n}(1+\vert \zeta\vert^2).\]
This metric expression is clearly regular down to $z=0$, in which case it reduces to 
\[\pi_i^*{g_{EH}}_{\vert z=0}=\frac{1}{a^{n-1}n^2}(1+\vert \zeta\vert^2)^n \vert dz\vert^2+ a g_{FS}.\]
To deduce \eqref{eq:EHPullback}, start by writing 
\[w_k\coloneqq \begin{cases} \sqrt[n]{z} \zeta_k & k\neq i\\ \sqrt[n]{z} & k=i.\end{cases}\]
One then finds
\[\vert dw\vert^2 \coloneqq \sum_{k=1}^n \vert dw_k\vert^2=\vert z\vert^{2/n}\left(\frac{(1+\vert \zeta\vert^2)}{n^2 \vert z\vert^2} \vert dz\vert^2 + \vert d\zeta\vert^2 +\frac{1}{n\vert z\vert^2}\sum_{k}\left(\overline{z}\zeta_k dz d\overline{\zeta}_k+z\overline{\zeta}_k d\overline{z} d\zeta_k\right)\right)\]
and
\[\overline{w}\cdot dw\coloneqq \sum_k \overline{w}_k dw_k=\vert z\vert^{2/n}\left(\frac{(1+\vert \zeta\vert^2)}{nz} dz+\sum_k\overline{\zeta}_k d\zeta_k\right).\]
Hence
\begin{align*}
&\vert dw\vert^2-\frac{a^n}{a^n+u^n}\frac{\vert \overline{w}\cdot dw\vert^2}{u} \\
&=\vert z\vert^{2/n}\Bigg(\frac{(1+\vert \zeta\vert^2)}{n^2\vert z\vert^2} \frac{u^n}{a^n+u^n} \vert dz\vert^2+\left(\vert d\zeta\vert^2-\frac{a^n}{a^n+u^n}\frac{\vert \overline{\zeta}\cdot d\zeta\vert^2}{1+\vert \zeta\vert^2}\right)\\
&+\frac{u^n}{n\vert z\vert^2 (a^n+u^n)}\sum_k\left(\left(\overline{z}\zeta_k dz d\overline{\zeta}_k+z\overline{\zeta}_k d\overline{z} d\zeta_k\right)\right)\Bigg).
\end{align*}
We now use $u=\sum_k \vert w_k\vert^2=\vert z\vert^{2/n}(1+\vert \zeta\vert^2)$ and
\[g_{FS}=\frac{(1+\vert \zeta\vert^2)\vert d\zeta\vert^2 -\vert \overline{\zeta}\cdot d\zeta\vert^2}{(1+\vert \zeta\vert^2)^2}\]
to deduce the statement.

\end{proof}

\begin{proof}[Proof of Proposition \ref{Prop:EHCurvature}]
The starting point will be equation \cite[Equation 4.39]{Kahler}), which says that
\begin{equation}
\Gamma_{\alpha\mu}^\lambda=g_{\mu\bar{\nu},\alpha}g^{\bar{\nu}\lambda}.
\notag
\end{equation}
We return to the setting of Lemma \ref{Lem:RotSymMetric} and introduce some shorthand notation,
\begin{equation}
e^{\psi(u)}\coloneqq f'(u)
\notag
\end{equation}
and 
\begin{equation}
\phi(u) \coloneqq -uf''(u)e^{-\psi(u)}.
\notag
\end{equation}
Lemma \ref{Lem:RotSymMetric} tells us that we may write
\begin{equation}
g_{\mu\bar{\nu}}=e^{\psi(u)}\left(\delta_{\mu\bar{\nu}}-\phi(u)\frac{\bar{z}_\mu z_\nu}{u}\right),
\label{eq:gPhiPsi}
\end{equation}
One readily checks that the inverse of $g$ is given as
\begin{equation}
g^{\bar{\nu}\lambda}=e^{-\psi(u)}\left(\delta^{\bar{\nu}\lambda}+\frac{\phi(u)}{1-\phi(u)}\frac{\bar{z}^\nu z^\lambda}{u}\right).
\label{eq:ginvPhiPsi}
\end{equation}

Furthermore, the definition of $e^\psi$ tells us that $\psi'(u)=-\frac{\phi(u)}{u}$. One can thus compute that
\begin{align}
g_{\mu\bar{\nu},\alpha}&=-\frac{\phi(u)}{u}g_{\mu\bar{\nu}} \bar{z}_\alpha+e^{\psi(u)}\left(\frac{\phi(u)-u\phi'(u)}{u^2} \bar{z}_\mu z_\nu\bar{z}_\alpha-\frac{\phi(u)}{u}\delta_{\alpha\bar{\nu}}\bar{z}_\mu\right)\notag \\
&=-\frac{\phi(u)}{u}(g_{\mu\bar{\nu}}\bar{z}_\alpha+g_{\alpha\bar{\nu}}\bar{z}_\mu)+e^{\psi(u)}\frac{\phi(u)(1-\phi(u))- u\phi'(u)}{u}\frac{\bar{z}_\mu z_\nu \bar{z}_\alpha}{u}.
\notag
\end{align}
Multiply this by $g^{\bar{\nu}\lambda}$ and sum over $\nu$, observing that $z_\nu$ is an eigenvector of $g^{\bar{\nu}\lambda}$ with eigenvalue $\frac{e^{-\psi}}{1-\phi}$. This leads to
\begin{equation}
\Gamma^\lambda_{\mu\alpha}=g_{\mu\bar{\nu},\alpha}g^{\bar{\nu}\lambda}=-\frac{\phi(u)}{u}(\delta_{\mu}{}^\lambda \bar{z}_\alpha +\delta_\alpha{}^\lambda \bar{z}_\mu)+\frac{\phi(u)(1-\phi(u))-u\phi'(u)}{u(1-\phi(u))}\cdot\frac{\bar{z}_\mu \bar{z}_\alpha z^\lambda}{u}.
\label{eq:GeneralRotSymChristoffel}
\end{equation}
Equation \eqref{eq:GeneralRotSymChristoffel} holds without assuming that the metric is Ricci-flat. One "only" needs to assume that the K\"{a}hler potential is spherically symmetric\footnote{Two special examples would be $e^\psi=1$, corresponding to the Euclidean metric, or $e^\psi=\frac{a}{a+u}$ and $\phi=\frac{u}{a+u}$, corresponding to the Fubini-Study metric.}. When we in addition assume that the metric is Ricci-flat, we have seen in Lemma \eqref{Lem:Ricci-FlatMetric} that 
\begin{equation}
\phi=\frac{a^n}{a^n+u^n},
\notag
\end{equation}
hence
\begin{equation}
\phi'(u)=-n\frac{u^{n-1}a^n}{(a^n+u^n)^2}=-\frac{n}{u} \phi(u)(1-\phi(u)).
\label{eq:EHPhiDerEq}
\end{equation}
 Inserting this into equation \eqref{eq:GeneralRotSymChristoffel} yields \eqref{eq:GeneralChristoffel}.

To compute the Riemann curvature, we will employ equation \cite[Equation 4.49]{Kahler}, which says that for K\"{a}hler manifolds, the curvature can be written as
\begin{equation}
R^{\lambda}{}_{\mu\bar{\beta}\alpha}=-\frac{\partial \Gamma^\lambda_{\mu\alpha}}{\partial \bar{z}^\beta}.
\notag
\end{equation}
Using this fact, one can differentiate \eqref{eq:GeneralChristoffel}, to find that as a first step,
\begin{align}
R^\lambda{}_{\mu\bar{\beta}\alpha}=& \left(\frac{\phi(u)}{u}\right)'z_\beta\left(\bar{z}_\mu \delta^\lambda{}_\alpha+\bar{z}_\alpha \delta^\lambda{}_\mu -(n+1)\frac{\bar{z}_\alpha\bar{z}_\mu}{u}z^\lambda\right)\notag \\
&+ \frac{\phi(u)}{u}\left(\delta_{\mu\bar{\beta}}\delta^\lambda{}_\alpha+\delta_{\alpha\bar{\beta}}\delta^\lambda{}_\mu-(n+1)\frac{z^\lambda}{u}\left(\delta_{\alpha\bar{\beta}}\bar{z}_\mu+\delta_{\mu\bar{\beta}}\bar{z}_\alpha-\frac{\bar{z}_\mu \bar{z}_\alpha z_\beta}{u}\right)\right).
\notag
\end{align}
Here we are still using the shorthand $\phi=\frac{a^n}{a^n+u^n}$.
Using \eqref{eq:EHPhiDerEq}, we find
\begin{equation}
\left(\frac{\phi(u)}{u}\right)'=\frac{u\phi'(u)-\phi}{u^2}=-n\frac{\phi(1-\phi)}{u^2}-\frac{\phi}{u^2},
\notag
\end{equation}
which allows us to write
\begin{align}
R^\lambda{}_{\mu\bar{\beta}\alpha}=&\frac{\phi(u)}{u}\Bigg(\left(\delta_{\mu\bar{\beta}}-\frac{\bar{z}_\mu z_\beta}{u}\right)\delta^\lambda{}_\alpha+\left(\delta_{\alpha\bar{\beta}}-\frac{\bar{z}_\alpha z_\beta}{u}\right)\delta^\lambda{}_\mu\notag \\
&-(n+1)\left(\frac{\bar{z}_\mu z^\lambda}{u}\delta_{\alpha\bar{\beta}}+\frac{\bar{z}_\alpha z^\lambda}{u}\delta_{\mu\bar{\beta}}-2\frac{\bar{z}_\mu z^\lambda \bar{z}_\alpha z_\beta}{u^2}\right)\notag\\
&-n(1-\phi)\left(\frac{\bar{z}_\mu z_\beta}{u} \delta^\lambda{}_\alpha+\frac{\bar{z}_\alpha z_\beta}{u}\delta^\lambda{}_\mu -(n+1)\frac{\bar{z}_\mu z^\lambda \bar{z}_\alpha z_\beta}{u^2}\right)\Bigg).
\notag
\end{align}
This can be rewritten by observing that $g_{\mu\bar{\nu}}e^{-\psi}=\delta_{\mu\bar{\nu}}-\phi \frac{\bar{z}_\mu z_\nu}{u}$, resulting in 

\begin{align}
R^\lambda{}_{\mu\bar{\beta}\alpha}=&\frac{\phi}{u}\Bigg(e^{-\psi}\left(g_{\mu\bar{\beta}}\delta^\lambda{}_\alpha+g_{\alpha\bar{\beta}}\delta^\lambda{}_\mu\right)\notag \\
&-(n+1)e^{-\psi}\left(\frac{\bar{z}_\mu z^\lambda}{u}g_{\alpha\bar{\beta}}+\frac{\bar{z}_\alpha z^\lambda}{u}g_{\mu\bar{\beta}}\right)\notag\\
&-(n+1)(1-\phi)\left(\frac{\bar{z}_\mu z_\beta}{u} \delta^\lambda{}_\alpha+\frac{\bar{z}_\alpha z_\beta}{u}\delta^\lambda{}_\mu\right)\notag \\
&+(n+1)(n+2)(1-\phi)\frac{\bar{z}_\mu z^\lambda \bar{z}_\alpha z_\beta}{u^2}\Bigg).
\label{eq:RiemannThirdStep}
\end{align}

Finally, we need a small observation. The inverse of the metric, expressed using $\phi$ and $\psi$, is given in equation \eqref{eq:ginvPhiPsi}. Using that equation, and recalling that we raise and lower indices on $z$ using the Euclidean metric in this article, one may write
\begin{equation}
z^\lambda =e^\psi(1-\phi)g^{\bar{\sigma}\lambda} z_\sigma.
\notag
\end{equation}
Inserting this into \eqref{eq:RiemannThirdStep}, we arrive at
\begin{align}
R^\lambda{}_{\mu\bar{\beta}\alpha}&=\frac{\phi}{u}e^{-\psi} g^{\bar{\sigma}\lambda}\Bigg((g_{\alpha\bar{\sigma}}g_{\mu\bar{\beta}}+g_{\mu\bar{\sigma}}g_{\alpha\bar{\beta}})\notag \\
&-e^{\psi}(1-\phi)(n+1)\left(\frac{\bar{z}_\mu z_\beta}{u}g_{\alpha\bar{\sigma}}+\frac{\bar{z}_\alpha z_\beta}{u}g_{\mu\bar{\sigma}}+\frac{\bar{z}_\mu z_\sigma}{u}g_{\alpha\bar{\beta}}+\frac{\bar{z}_\alpha z_\sigma}{u}g_{\mu\bar{\beta}}\right)\notag \\
&+e^{2\psi}(1-\phi)^2(n+1)(n+2)\frac{\bar{z}_\mu z_\sigma \bar{z}_\alpha z_\beta}{u^2}\Bigg).
\notag
\end{align}
This is almost \eqref{eq:EHRiemannTensor}. What remains is to multiply by $g_{\sigma\bar{\nu}}$, sum over $\sigma$, insert $\phi(u)=\frac{a^n}{a^n+u^n}$ and $\exp(\psi)=\sqrt[n]{1+\frac{a^n}{u^n}}$, and observe that $R_{\mu\bar{\nu}\alpha\bar{\beta}}=R_{\bar{\nu}\mu \bar{\beta}\alpha}$.

To compute the Kretschmann scalar, start with \eqref{eq:EHRiemannTensor} and raise the indices using $g$. Then multiply the result with \eqref{eq:EHRiemannTensor} and contract. This gives 6 different kinds of terms. Adding them up gives \eqref{eq:EHKretschmann}.

\end{proof}

\begin{Rem}
As a test, one can check directly by summing over $\lambda=\alpha$ or $\lambda=\mu$ in \eqref{eq:RiemannThirdStep} that the Ricci-tensor vanishes identically, as it should.  
\end{Rem}

\begin{proof}[Proof of Proposition \ref{Prop:HolomForm}]
We have shown in Proposition \ref{Prop:HolomFormExtends} that $\pi^* \eta^\C$ extends to a nowhere vanishing holomorphic form on $\ehn$, so we only need to argue that it is parallel. We may work on $\C^n\setminus \{0\}$ with the metric given by \eqref{eq:GeneralEHMetric} One way to argue is to use the Christoffel symbols \eqref{eq:GeneralChristoffel} and simply compute. Indeed, $\eta^{\C}=\frac{1}{n!} \epsilon_{\mu_1,\dots,\mu_n} dz^{\mu_1} \wedge \dots \wedge dz^{\mu_n}$ where $\epsilon$ denotes the completely anti-symmetric Levi-Civita symbol and $\epsilon_{1,2,\dots,n}=1$. Then
\[\nabla_\alpha \epsilon_{\mu_1,\dots,\mu_n}=-\Gamma^\lambda_{\alpha\mu_1} \epsilon_{\lambda,\mu_2,\dots,\mu_n}-\dots -\Gamma^{\lambda}_{\alpha \mu_n} \epsilon_{\mu_1,\dots,\lambda}.\]
Due to the anti-symmetry, it suffices to compute with $(\mu_1,\dots ,\mu_n)=(1,2,\dots,n)$. We then have
\begin{align*}
&-\Gamma^\lambda_{\alpha\mu_1} \epsilon_{\lambda,\mu_2,\dots,\mu_n}-\dots -\Gamma^{\lambda}_{\alpha \mu_n} \epsilon_{\mu_1,\dots,\lambda}\\
&=\frac{a^n}{u(a^n+u^n)}\left( n\overline{z}_\alpha+(\overline{z}_1 \delta_{\alpha 1}+\dots+\overline{z}_n \delta_{\alpha n})-(n+1)\frac{\overline{z}_\alpha}{u}(\overline{z}_1 z^\lambda \delta_{\lambda 1}+\dots+ \overline{z}^n z^\lambda \delta_{\lambda n}\right)\\
&=\frac{a^n}{u(a^n+u^n)}\left((n+1)\overline{z}_\alpha -(n+1)\overline{z}_\alpha\right)=0.
\end{align*}

An alternative argument is first to observe $\vert \eta^{\C}\vert^2_g = \frac{\det(g)}{n!}=\frac{1}{n!}$ is constant. If $\tilde{\eta}$ is the parallel, holomorphic $n$-form on $\ehn$, then $\tilde{\eta}=f\cdot \eta$ for some holomorphic function $f$. But a parallel form has a constant norm, hence $\vert f\vert$ is constant. So $f$ is a constant and $\eta$ is parallel.
\end{proof}

\subsection{Minimal submanifolds}
Before we prove Theorem \ref{Thm:Minimal}, we need to state a well-known result about minimal submanifolds.
\begin{Lem}[{\cite[Lemma 5.1]{MeanCurv}}]
\label{Lem:Hess}
Let $(M,g)$ be a Riemannian manifold. Suppose that $\varphi$ is a smooth function on $M$ whose Hessian is positive semidefinite. Then any compact, minimal submanifold $K$ of $M$ must be contained in the set where $Hess_M \varphi$ degenerates. In addition, $\varphi$ is constant on connected components of $K$.
\end{Lem}
We refer to \cite{MeanCurv} for the proof. As Tsai and Wang point out, the lemma is older and not originally due to them, but neither they nor the author know of another reference with a proof.

\begin{proof}[Proof of Theorem \ref{Thm:Minimal}]
As in \cite{MeanCurv}, we choose $\varphi$ to be the distance squared to the zero section $\C\P^{n-1}\subset \mathcal{O}_{\C\P^{n-1}}(-n)$, and by Lemma \ref{Lem:Hess} we just need to show that the Hessian of $\varphi$ is positive definite at any $z\in \left(\C^n\setminus  \{0\}\right)/\mu_n$. We may compute on the $n$-fold cover $\C^n\setminus \{0\}$. Let $\gamma(t)=\alpha(t)z$ be a radial geodesic, where $\alpha(0)=0$, $\alpha(t), \dot{\alpha}(t)>0$ for $t\in (0,1]$, and\footnote{One can in principle find $\alpha$ by solving the equation $\vert \dot{\gamma}\vert^2_g=C^2=const.$, which combines with \eqref{eq:GeneralEHMetric} to yield
$\dot{\alpha}(t)=C\left(\frac{a^n}{\alpha(t)^{2n}}+\vert z\vert_{\C^n}^{2n}\right)^{\frac{n-1}{2n}},$
but we will not need the explicit solution.} $\alpha(1)=1$.  
By \eqref{eq:GeneralEHMetric} we find
\[\vert \dot{\gamma}\vert^2_g=u\dot{\alpha}(t)^2 \left(u\frac{\alpha(t)^2}{\sqrt[n]{\alpha(t)^{2n}u^{n}+ a^n}}\right)^{n-1}\]
and
\begin{align}\sqrt{\varphi}(z)&=\int_0^1 \vert \dot{\gamma}\vert_g\, dt = \sqrt{u}\int_0^1 \dot{\alpha}(t) \left(\frac{u\alpha(t)^2}{\sqrt[n]{\alpha(t)^{2n} u^{n}+ a^n}}\right)^{\frac{n-1}{2}}\, dt \notag \\
&=\sqrt{u}\int_0^1 \left(\frac{u\alpha^2}{\sqrt[n]{\alpha^{2n}u^{n}+ a^n}}\right)^{\frac{n-1}{2}}\, d\alpha \notag \\
&=\frac{\sqrt{a}}{n} \int_0^{\left(\frac{u}{a}\right)^{\frac{n}{2}}} \frac{1}{(\tau^{2}+1)^{\frac{n-1}{2n}}}\, d\tau \notag\\
&\eqqcolon \sqrt{\psi}(u(z)). 
\label{eq:length} 
\end{align}
where we have substituted $\tau=\left(\sqrt{\frac{u}{a}}\alpha\right)^n$. 

We use first compute the components of the Hessian when $\C$-linearly extended to the complexified tangent bundle. Writing $H\coloneqq \nabla d\varphi$, we have 4 block matrices to consider. They read
\begin{equation}
H_{\mu\oa}\coloneqq H\left(\partial_{z^{\mu}},\partial_{\overline{z}^{\alpha}}\right)=\partial_{\mu}\partial_{\oa}\varphi=\delta_{\mu\alpha} \psi'(u)+\overline{z}_\mu z_\alpha \psi''(u),
\label{eq:MixedHessian}
\end{equation}
\[ H_{\mu\alpha}\coloneqq H\left(\partial_{z^{\mu}},\partial_{z^{\alpha}}\right)=\partial_{\mu}\partial_{\alpha}\varphi-\Gamma^{\lambda}_{\mu\alpha} \partial_{\lambda} \varphi=\overline{z}_\mu \overline{z}_\alpha \psi''(u)-\Gamma_{\mu\alpha}^{\lambda} \overline{z}_{\lambda} \psi'(u),\]
\[H_{\oa\mu}\coloneqq H\left(\partial_{\overline{z}^{\alpha}},\partial_{z^{\mu}}\right)=\overline{H_{\alpha\om}},\]
\[H_{\om\oa}\coloneqq H\left(\partial_{\overline{z}^{\mu}},\partial_{\overline{z}^{\alpha}}\right)=\overline{H_{\mu\alpha}}.\]
The Christoffel symbols for the Calabi-Eguchi-Hanson metric \eqref{eq:GeneralEHMetric} are given by \eqref{eq:GeneralChristoffel}. From that equation, we easily find
\[\Gamma_{\mu\alpha}^{\lambda} \overline{z}_{\lambda}=(n-1)\frac{a^n}{u(a^n+u^n)}\overline{z}_\mu \overline{z}_\alpha\eqqcolon \Upsilon(u) \overline{z}_\mu \overline{z}_\alpha.\]
So
\begin{equation}H_{\mu\alpha}=\left(\psi''(u)-\Upsilon(u)\psi'(u)\right)\overline{z}_\mu\overline{z}_{\alpha}.
\label{eq:HolomHessian}
\end{equation}
To translate into the components of the real Hessian, we write $z^\mu=x^\mu +i y^\mu$ and use 
\[\frac{\partial}{\partial z^\mu}=\frac{1}{2}\left(\frac{\partial}{\partial x^\mu}-i\frac{\partial}{\partial y^\mu}\right).\]
to write
\begin{equation}
H^{xx}_{\mu\alpha}\coloneqq H(\partial_{x^\mu},\partial_{x^\alpha})=H(\partial_{z^\mu}+\partial_{\overline{z}^\mu},\partial_{z^\alpha}+\partial_{\overline{z}^\alpha}
)=H_{\mu\alpha}+H_{\om\oa}+H_{\om \alpha}+H_{\mu\oa}.
\label{eq:Hessxx}
\end{equation}
Similarly,
\begin{equation}
H^{xy}_{\mu\alpha}\coloneqq H(\partial_{x^\mu},\partial_{y^\alpha})=H\left(\partial_{z^\mu}+\partial_{\overline{z}^\mu},i\left(\partial_{z^\alpha}-\partial_{\overline{z}^\alpha}\right)
\right)=i(H_{\mu\alpha}-H_{\om\oa})+i(H_{\om \alpha}-H_{\mu\oa}),
\label{eq:Hessxy}
\end{equation}
\begin{equation}
H^{yx}_{\mu\alpha}\coloneqq H(\partial_{y^\mu},\partial_{x^\alpha})=H\left(i\left(\partial_{z^\mu}-\partial_{\overline{z}^\mu}\right),\partial_{z^\alpha}+\partial_{\overline{z}^\alpha}
\right)=i(H_{\mu\alpha}-H_{\om\oa})+i(H_{\mu\oa}-H_{\om \alpha}),
\label{eq:Hessyx}
\end{equation}
and
\begin{equation}
H^{yy}_{\mu\alpha}\coloneqq H(\partial_{y^\mu},\partial_{y^\alpha})=H\left(i\left(\partial_{z^\mu}-\partial_{\overline{z}^\mu}\right),i\left(\partial_{z^\alpha}-\partial_{\overline{z}^\alpha}\right)
\right)=-(H_{\mu\alpha}+H_{\om\oa})+(H_{\om \alpha}+H_{\mu\oa}).
\label{eq:Hessyy}
\end{equation}
The Hessian thus takes the block-matrix form
\[\begin{pmatrix}
H^{xx} & H^{xy}\\ H^{yx} & H^{yy}\end{pmatrix}\]
in these coordinates.

 By \eqref{eq:Hessxx}, \eqref{eq:MixedHessian} and \eqref{eq:HolomHessian}, we find\footnote{Recall that $x\otimes x$ means the matrix with entries $(x\otimes x)_{k\ell}= x_k x_{\ell}$.}
\[H^{xx}=2\psi'\mathds{1}_{n} +2(2\psi''-\Upsilon \psi')x\otimes x+ 2\Upsilon \psi' y\otimes y.\]
Similarly, we have 
\[H^{yy}=2\psi' \mathds{1}_{n}+2\Upsilon\psi' x\otimes x+2(2\psi''-\Upsilon\psi')y\otimes y,\]
\[H^{xy}=2(2\psi''-\Upsilon \psi')x\otimes y-2\Upsilon \psi' y\otimes x,\]
and 
\[H^{yx}=2(2\psi''-\Upsilon \psi')y\otimes x-2\Upsilon \psi' x\otimes y.\]
Writing
\[A\coloneqq 2\frac{\psi''}{\psi'}-\Upsilon\]
\[B\coloneqq \Upsilon,\]
the real Hessian takes the form
\[H=2\psi'\left(\mathds{1}_{2n} +\begin{pmatrix}
A x\otimes x+B y\otimes y & A x\otimes y - B y\otimes x\\ A y\otimes x - B x\otimes y & B x\otimes x + A y\otimes y\end{pmatrix}\right).\]
The eigenvalues of the matrix
\[\begin{pmatrix}
A x\otimes x+B y\otimes y & A x\otimes y - B y\otimes x\\ A y\otimes x - B x\otimes y & B x\otimes x + A y\otimes y\end{pmatrix}\]
are given by 
\begin{itemize}
\item 0 with eigenvector $V=\begin{pmatrix} V_1\\ V_2\end{pmatrix}$ with $V_1,V_2$ orthogonal to both $x$ and $y$;
\item $A(\vert x\vert^2+\vert y\vert^2)=Au$ with eigenvector $\begin{pmatrix} x\\ y\end{pmatrix}$;
\item  $Bu$ with eigenvector $\begin{pmatrix}
y\\ -x
\end{pmatrix}$.
\end{itemize}
So the eigenvalues of the real Hessian $H$ are 
\begin{itemize}
\item $\lambda_1=2\psi'$ with multiplicity $2n-2$;
\item $\lambda_2=2\psi'+2\psi'Au=2\psi'+2u(2\psi''-\Upsilon \psi')$ with multiplicity 1;
\item $ \lambda_3=2\psi'+2\psi'Bu=2\psi'(1+u\Upsilon)$ with multiplicity 1.
\end{itemize}
 We therefore need to check that these three eigenvalues are positive for $u>0$.

From the explicit expression for the length squared \eqref{eq:length} we find
\[\psi'(u)=\frac{\sqrt{\psi}}{\sqrt{u}}\left(\frac{u^n}{a^n+u^n}\right)^{\frac{n-1}{2n}}.\]
From this and the definition of $\Upsilon=(n-1)\frac{a^n}{u(a^n+u^n)}$, the eigenvalues $\lambda_1$ and $\lambda_3$ are positive.
Differentiating again, a straight forward computation gives
\[\psi''(u)=\frac{\psi'}{2\psi}\left(\psi'+\frac{\psi}{u(a^n+u^n)}\left((n-2)a^n-u^n\right)\right).\]
So
\[2\psi''-\Upsilon\psi'=\frac{\psi'}{u\psi}(u\psi'-\psi)=\frac{(\psi')^2}{\psi}-\frac{\psi'}{u}.\]
Hence
$\lambda_2=2\psi' +2(2\psi''-\Upsilon\psi')u=2u\frac{(\psi')^2}{\psi}>0$.

\end{proof}

Geodesics are of course minimal submanifolds, so the corollary follows. For completeness' sake, we offer an alternative, direct proof. The spirit is the same as the above proof.
\begin{proof}[Proof of Corollary \ref{Cor:NoClosedEH}]
Let us first prove that $\C\P^{n-1}$ sits inside $\mathcal{O}_{\C\P^{n-1}}(-n)$ as a totally geodesic submanifold. Consider a diagonal $U(1)$-action on $\C^n$, $(\zeta,z)\mapsto \zeta z$. According to Lemma \ref{Lem:GActiononCn}, we get an induced $U(1)$-action on $\mathcal{O}_{\C\P^{n-1}}(-n)$. This $U(1)$-action is isometric (since it obviously preserves \eqref{eq:GeneralEHMetric} and \eqref{eq:GeneralFSMetric}). For $\zeta\in U(1)$, $\zeta\notin \mu_n$, no point of $\left(\C^n\setminus\{0\}\right)/\mu_n$ is left fixed, but all of $\C\P^{n-1}$ is fixed. So $Fix(\zeta)=\C\P^{n-1}$. and  $\C\P^{n-1}\subset \mathcal{O}_{\C\P^{n-1}}(-n)$ is totally geodesic.

To prove the corollary, assume $\gamma\colon [0,1]\to \mathcal{O}_{\C\P^{n-1}}(-n)$ is a closed geodesic. There are two cases. In the first case $\gamma(t)\in \C\P^{n-1}$ for all $t$ in some interval. Then $\gamma$ is a geodesic on $\C\P^{n-1}$ with the Fubini-Study metric by Theorem \ref{Thm:EHMetric}, and hence remains on $\C\P^{n-1}$ for all time. Then one quotes for instance\footnote{This is according to Klingenberg \cite[p.178]{Kli78} due to Cartan. \cite[Prop. 5.3, p. 356]{Helgason} has a proof.} \cite[Theorem 5.2.1, p. 178]{Kli78} to say that all geodesics (when their domain is extended to all of $\R$) in $\C\P^{n-1}$ with the Fubini-Study metric are closed. The stability statement comes from the fact that any geodesic in $(\C\P^{n-1},g_{FS})$ lies on some $\S^2\cong \C\P^1\subset \C\P^{n-1}$ (this statement is also part of \cite[Theorem 5.2.1, p. 178]{Kli78}), and is therefore not stable.

The second case is that $\gamma([0,1])$ is not contained in $\C\P^{n-1}$. Then there has to be some maximal interval $(t_0,t_1)\subset [0,1]$ with $\gamma(t)\in \left(\C^n\setminus \{0\}\right)/\mu_n$. Let $u(t)\coloneqq \vert \gamma(t)\vert^2_{\C^n}$ for $t\in (t_0,t_1)$. Since $\gamma$ is closed, $u(t)$ has to have a maximum, meaning there is some $T\in(t_0,t_1)$ such that $\dot{u}(T)=0$ and $\ddot{u}(T)\leq 0$. This will lead to a contradiction unless $u(t)$ is constant, as we explain next. On $\left(\C^n\setminus \{0\}\right)/\mu_n$, the Christoffel symbols of the metric are given by \eqref{eq:GeneralChristoffel}. The geodesic equation on a K\"{a}hler manifold reads\footnote{To see this, take the usual geodesic equation and observe that $\Gamma_{\mu\bar{\alpha}}^\lambda=\Gamma_{\bar{\mu}\alpha}^\lambda=0$ for a K\"{a}hler manifold.}
\begin{equation}
\ddot{\gamma}(t)^\lambda+\Gamma^\lambda_{\mu\alpha}\dot{\gamma}(t)^\alpha \dot{\gamma}(t)^\mu=0.
\notag
\end{equation}
In local coordinates with $z(t)\coloneqq \gamma(t)$, we get
\begin{equation}
\ddot{z}^\lambda=\frac{a^n}{a^n+u^n}\left(2\frac{\ip{z}{\dot{z}}}{u}\dot{z}^\lambda -(n+1)\frac{\ip{z}{\dot{z}}^2}{u^2}z^\lambda\right)
\label{eq:EH-Geodesic-Eq}
\end{equation}
upon inserting \eqref{eq:GeneralChristoffel} (the inner products in \eqref{eq:EH-Geodesic-Eq} are all $\C^n$ inner products).

Directly from the definition of $u(t)$, we have
\begin{equation}
\dot{u}(t)=\ip{z}{\dot{z}}(t)+\ip{\dot{z}}{z}(t)=2\text{Re}\left(\ip{z}{\dot{z}}\right)(t),
\notag
\end{equation}
hence $\dot{u}(T)=0$ if and only if $ \ip{z}{\dot{z}}(T)=i\alpha \cdot u(T)$ for some $\alpha\in \R$ (the extra factor of $u(T)$ is just for later convenience). Inserting $\ip{z}{\dot{z}}(T)=i\alpha \cdot u(T)$ into \eqref{eq:EH-Geodesic-Eq} gives us
\begin{equation}
\ddot{z}^\lambda(T)=\frac{a^n}{a^n+u(T)^n}\left(2i\alpha \dot{z}(T)^\lambda+(n+1)\alpha^2 z(T)^\lambda\right).
\label{eq:SpecialGeodesicEq}
\end{equation}

Consider now the second derivative of $u$,
\begin{equation}
\ddot{u}(t)=2\text{Re}\left(\ip{z}{\ddot{z}}\right)(t)+2\vert \dot{z}\vert^2_{\C^n}(t).
\notag
\end{equation}
At the time $T$, we can use the expression for $\ddot{z}(T)$ in \eqref{eq:SpecialGeodesicEq} to deduce
\begin{align}
\ip{z}{\ddot{z}}(T)&=\frac{a^n}{a^n+u(T)^n}\left(2i\alpha \ip{z}{\dot{z}}(T)+(n+1)\alpha^2 u(T)\right)\notag \\
 &=(n-1)\frac{a^n}{a^n+u(T)^n}u(T)\alpha^2,
 \notag
\end{align}
hence
\begin{equation}
\ddot{u}(T)=2(n-1)\frac{a^n}{a^n+u(T)^n}u(T)\alpha^2 +2\vert \dot{z}\vert^2_{\C^n}(T),
\label{eq:uTdder}
\end{equation}
where both terms are non-negative, ensuring $\ddot{u}(T)\geq 0$. Since we're demanding that $\ddot{u}(T)\leq 0$, we deduce $\ddot{u}(T)=0$.  Referring back to \eqref{eq:uTdder}, we see that $\dot{z}(T)=0$. One can now conclude that we have a geodesic with vanishing velocity at a point, hence, by uniqueness of geodesics, we have a constant geodesic. So there are no non-constant, closed geodesic which are not contained in $\C\P^{n-1}$.

\end{proof}

\subsection{Homothety}
The Calabi-Eguchi-Hanson metrics have an interesting scaling property. Let $\alpha>0$ be any fixed number, and consider the homothety $h_{\alpha}^{\C} \colon \C^n\to \C^n$ given by $h_{\alpha}^{\C}(z)=\alpha z$. This clearly fixes the origin and commutes with both the $\mu_k$-action and the blow-up, giving us a map $h_{\alpha}\colon \mathcal{O}_{\C\P^{n-1}}(-k)\to \mathcal{O}_{\C\P^{n-1}}(-k)$.
\begin{Lem}
Let $g_a$ denote the Calabi-Eguchi-Hanson metric on $\mathcal{O}_{\C\P^{n-1}}(-n)$ with parameter $a>0$. Let $\alpha>0$. Then
\[h_{\alpha}^* g_{\alpha^2 a}=\alpha^2 g_{a}.\]
\end{Lem}
\begin{proof}
Either observe that the potential \eqref{eq:Calabi-Pot} satisfies $f_{\alpha^2 a}(\alpha z)=\alpha^2 f_a(z)$, or consider how the metric \eqref{eq:Ricci-FlatMetric} changes.
\end{proof}


\begin{thebibliography}{99}









\bibitem[\textsc{Abh98}]{Abhyankar}
Shreeram S. Abhyankar,
\emph{Resolution of {Singularities} of {Embedded} {Algebraic} {Surfaces}}, Springer, (1998).



\bibitem[\textsc{Bal06}]{Kahler}
Werner Ballmann, \textit{Lectures on K\"{a}hler Manifolds}, EMS, (2006).

\bibitem[\textsc{BHPV04}]{Surfaces}
Wolf P. Barth, Klaus Hulek,  Chris A. M. Peters, and Antonius Van de Ven,
\emph{Compact Complex Surfaces, 2nd Ed.},
Springer, (2004).

\bibitem[\textsc{Bie02}]{Bielawski}
Roger Bielawski, \emph{Ricci-flat K\"{a}hler metrics on canonical bundles},  Math. Proc. Cambridge Phil. Soc. 132, pp. 471--479, (2002).

\bibitem[\textsc{BoYa73}]{BY73}
Jean-Pierre Bourguignon  and Shing-Tung Yau, \emph{Sur les m\'{e}triques riemanniens \`{a} courbure de {Ricci} nulle sur le quotient d'une surface {K3}}, C.R. Acad. Sc. Paris, t.277, Series A, pp. 1175-1177 (1973).

  \bibitem[\textsc{BrSa89}]{Bryant}
Robert L. Bryant and Simon M. Salamon, \emph{On the construction of some complete metrics with exceptional holonomy}, Duke Math. J. 58, no. 3, pp. 829--850, (1989).

\bibitem[\textsc{Cal79}]{Cal79}
Eugenio Calabi,
\emph{Metriques kähleriennes et fibres holomorphes},
Annales scientifique de l'E.N.S. 4e serie, tome 12, no. 2,
(1979).


\bibitem[\textsc{EgHa79}]{EH}
 Tohru Eguchi and Andrew J.  Hanson, \emph{Self-dual solutions to Euclidean gravity}, Annals of Physics 120, pp. 82–-105, (1979). 

\bibitem[\textsc{Fei01}]{Feix}
Birte Feix, \emph{Hyperkähler metrics on cotangent bundles}, J. reine angew. Math. vol. no. 532, pp. 33-46, (2001).

\bibitem[\textsc{GaDo13}]{DG13} 
Peng Gao. and Michael Douglas  
\emph{Geodesics on Calabi-Yau manifolds and winding states in non-linear sigma models},
Front. Phys., 16 December (2013) 

\bibitem[\textsc{GiPo79}]{GibbonsPope}
Garry W. Gibbons and Christopher N. Pope, \emph{The positive action conjecture and asymptotically Euclidean metrics in quantum gravity},  Comm. Math. Phys. 66(3), pp. 267--290 (1979).
 
 \bibitem[\textsc{Hel62}]{Helgason}
Sigurdur Helgason, Differential Geometry and Symmetric Spaces,
Academic Press, (1962).

\bibitem[\textsc{Hir64}]{Hironaka}
Heisuke Hironaka, \emph{Resolution of Singularities of an Algebraic Variety Over a Field of Characteristic Zero: I}, Annals of Mathematics Second Series, Vol. 79, No. 1, pp. 109-203, (1964).

\bibitem[\textsc{Huy05}]{Huybook}
Daniel Huybrechts, \emph{Complex Geometry}, Springer, (2005).

\bibitem[\textsc{Joy01}]{Joyce}
Dominic Joyce,
\emph{Asymptotically Locally {Euclidean} metrics with holonomy {SU(m)}},
Annals of Global Analysis and Geometry 19, pp. 55-73,
(2001)


\bibitem[\textsc{Kal00}]{Kaledin}
Dmitry Kaledin, \emph{Hyperk\"{a}hler structures on total spaces of holomorphic cotangent bundles}, in D. Kaledin, M. Verbitsky,
New constructions of hyperkahler manifolds, Intl. Press, Cambridge, MA, (2000).

\bibitem[\textsc{Kli78}]{Kli78}
Wilhelm Klingenberg,
\emph{Lectures on {Closed} {Geodesics}},
Springer,
(1978)




\bibitem[\textsc{Kob90}]{Kob}
Ryoichi Kobayashi \emph{Moduli of Einstein Metrics on a K3 Surface and Degeneration of Type I},
Advanced Studies in Pure Mathematics, 18.2, pp. 257--311 (1990).

\bibitem[\textsc{KoNo69}]{KoNo}
Shoshichi Kobayashi and Katsumi Nomizu,
\emph{{Foundations} of {Differential Geometry}, {Volume II}},
Wiley-Interscience, (1969).

\bibitem[\textsc{Kro89A}]{Kronheimer1}
Peter B. Kronheimer,
\emph{The construction of {ALE} spaces as hyper-{Kähler} quotients},
J.  Differential   Geometry,  29, pp. 665--683,
(1989).

\bibitem[\textsc{Kro89B}]{Kronheimer2}
Peter B. Kronheimer,
\emph{{A} {Torelli-type} theorem for gravitational instantons},
J.  Differential   Geometry,  29, pp. 685--697,
(1989).





\bibitem[\textsc{Lu99}]{Lu}
Peng Lu, \emph{K\"{a}hler-Einstein Metrics on Kummer Threefold and Special Lagrangian Tori}, Communications in Analysis and Geometry, Volume 7, Nr. 4, pp. 787--806 (1999).


\bibitem[\textsc{Lye19}]{PhD}
J\o rgen Olsen Lye, \emph{Stable Geodesics on a K3 Surface}, unpublished PhD thesis, (2019).

\bibitem[\textsc{Lye23}]{KummerPaper}
J\o rgen Olsen Lye, \emph{Geodesics on a K3 Surface near the Orbifold Limit}, Annals of Global Analysis and Geometry, volume 63, 20 (2023).

\bibitem[\textsc{Muk03}]{Mum03}
Shigeru Mukai, \emph{{An} {Introduction} to {Invariants} and {Moduli}},
Cambridge University Press, (2003).


\bibitem[\textsc{Pag78}]{Page}
Don N. Page, \emph{A physical picture of the K3 gravitational instanton}, Physics Letters B
Volume 80, Issues 1–2, pp. 55--57 (1978) 

\bibitem[\textsc{Ste93}]{Stenzel}
Matthew B. Stenzel, \emph{Ricci-flat metrics on the complexification of a compact rank one symmetric space}, Manuscripta Math. 80, no. 2, pp. 151--163, (1993).

\bibitem[\textsc{TsWa18}]{MeanCurv}
Chung-Jun Tsai and Mu-Tao Wang,
\emph{Mean curvature flows in manifolds of special holonomy},
Journal of  Differential Geometry, Volume 108, Number 3,  pp. 531--569, (2018).

\bibitem[\textsc{Yau78}]{Yau78}
Shing-Tung Yau,
\emph{On the {Ricci} {Curvature} of a {Compact} {K\"{a}hler} {Manifold} and the {Complex} {Monge-Amp\`{e}re} {Equation}, {I}},
Communications on Pure and Applied Mathematics, 31 (3), pp. 339-411, (1978).


\bibitem[\textsc{Wlp08}]{Wlod}
Jarek Włodarczyk, 
\emph{Resolution of singularities of analytic spaces}, Proceedings of {Gökova} {Geometry-Topology} {Conference}, (2008).

\end{thebibliography}
\end{document}